\documentclass[letterpaper,9pt]{llncs}
\usepackage{graphicx,psfrag}
\usepackage{amsmath,amssymb}
\usepackage{enumerate,enumitem}
\usepackage{url}
\let\doendproof\endproof
\renewcommand\endproof{~\hfill\qed\doendproof}
\begin{document}
\title{On the bend-number of planar and outerplanar graphs}
\titlerunning{planar bend number}


\author{Daniel Heldt\inst{1} \and Kolja Knauer\inst{1} \and Torsten Ueckerdt\inst{2} }

\authorrunning{Heldt, Knauer, Ueckerdt}

\institute{TU Berlin,   Berlin, Germany\\
\email{dheldt@math.tu-berlin.de}, \email{knauer@math.tu-berlin.de}\\ 
\and Charles University in Prague, Prague, Czech Republic\\
\email{ueckerdt@googlemail.com}\\
}
\maketitle

\begin{abstract}
The \emph{bend-number} $b(G)$ of a graph $G$ is the minimum $k$ such that $G$ may be represented as the edge intersection graph of a set of grid paths with at most $k$ bends. We confirm a conjecture of Biedl and Stern showing that the maximum bend-number of outerplanar graphs is~$2$. Moreover we improve the formerly known lower and upper bound for the maximum bend-number of planar graphs from $2$ and $5$ to $3$ and~$4$, respectively.
\end{abstract}

\section{Introduction}
\label{sec:int}

In 2007 Golumbic, Lipshteyn and Stern defined an \emph{EPG}\footnote{EPG stands for \emph{edge intersection graph of paths in the grid}} representation of a simple graph $G$ as an assignment of paths in the rectangular plane grid to the vertices of $G$, such that two vertices are adjacent if and only if the corresponding paths intersect in at least one grid edge, see~\cite{Gol-09}. EPG representations arise from VLSI grid layout problems~\cite{Bra-90} and as generalizations of \emph{edge-intersection graphs of paths on degree 4 trees}~\cite{Gol-08}. In the same paper Golumbic et al. show that every graph has an EPG representation and propose to restrict the number of bends per path in the representation. There has been some work related to this, see~\cite{Asi-09,Bie-10,Gol-09,Hel-10,Rie-09}.
\begin{figure}
 \centering
\psfrag{a}{\small a}
\psfrag{b}{\small b}
\psfrag{c}{\small c}
\psfrag{d}{\small d}
\psfrag{e}{\small e}
\psfrag{f}{\small f}
\psfrag{g}{\small g}
\psfrag{h}{\small h}
\psfrag{i}{\small i}
\psfrag{j}{\small j}
\psfrag{k}{\small k}
\psfrag{l}{\small l}
 \includegraphics{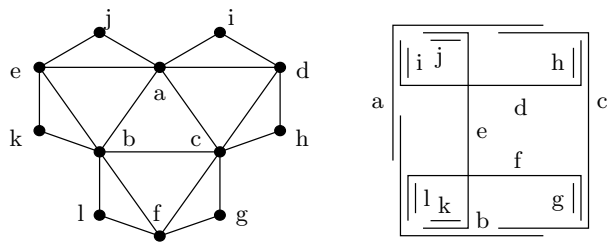}
 \caption{A $2$-bend graph and an EPG representation. If a grid edge is shared by several paths we draw them close to each other.}
 \label{fig:out-tight}
\end{figure}
A graph is a \emph{$k$-bend graph} if it has an EPG representation, where each path has at most $k$ bends. 
The \emph{bend-number} $b(G)$ of $G$ is the minimum $k$, such that $G$ is a $k$-bend graph.

%

Note that the class of $0$-bend graphs coincides with the well-known class of interval graphs, i.e., intersection graphs of intervals on a real line. 
It is thus natural to view $k$-bend graphs as an extension of the concept of interval graphs and the bend-number as a measure of how far a graph is from being an interval-graph. 

Intersection graphs of systems of intervals together with a parameter counting how many intervals are needed to represent a vertex of a given graph $G$ have received some attention. Popular examples are the \emph{interval-number}, see Harary and Trotter~\cite{Har-79} and the track-number, see Gy{\'a}rf{\'a}s and West~\cite{Gya-95}. Extremal questions for these parameters like \textit{`What is the maximum interval-/track-number among all graphs of a particular graph class?'} have been of strong interest in the literature. Scheinermann and West~\cite{Sch-83} show that the maximum interval-number of outerplanar graphs is $2$ and of planar graphs is $3$. Kostochka and West~\cite{Kos-99} prove that the maximum track-number of outerplanar graphs is $2$ and Gon{\c{c}}alves and Ochem~\cite{Gon-05} prove that the maximum track-number of planar graphs is $4$.

In~\cite{Bie-10} Biedl and Stern show that outerplanar graphs are $3$-bend graphs and provide an outerplanar graph which has bend-number $2$, see Fig.~\ref{fig:out-tight}. They conjecture that all outerplanar graphs are $2$-bend graphs. We confirm this conjecture in Theorem~\ref{bendThm:tw2-up}, showing the stronger result that graphs of treewidth at most $2$ are $2$-bend graphs.

The major part of this paper is devoted to planar graphs. Biedl and Stern~\cite{Bie-10} show that planar graphs are $5$-bend graphs but the only lower bound that they have is $2$, given by the graph of Fig.~\ref{fig:out-tight}. In Proposition~\ref{bendLem:stw3-low} we provide a planar graph of treewidth $3$ which has bend-number $3$, thus improving the lower bound of the class of planar graphs by one. Indeed in Theorem~\ref{bendThm:stw3-up} we show that every planar graph with treewidth $3$ is a $3$-bend graph. The main result of this article is Theorem~\ref{bendThm:pla-up}. We improve the upper bound for the bend-number of general planar graphs from $5$ to $4$.

\section{Preliminaries}
We consider simple undirected graphs $G$ with vertex set $V(G)$ as well as edge set $E(G)$. 
An EPG representation is a set of finite paths $\{P(u) \mid u \in V(G)\}$, which consist of consecutive edges of the rectangular grid in the plane. These paths are not self-intersecting (in grid edges), and $P(u) \cap P(v) \neq \emptyset$ if and only if $\{u,v\} \in E(G)$, i.e., only intersections at grid \emph{edges} are considered.
A bend of $P(u)$ is either a horizontal grid edge followed by a vertical one or a vertical edge followed by a horizontal one in $P(u)$. 

Let us now introduce some terminology: 
The grid edges between two consecutive bends (or the first (last) bend and the start (end) of $P(u)$) are called \emph{segments}. So
 a $k$-bend path consists of $k+1$ segments, each of which is either horizontal or vertical. A \emph{sub-segment} is a connected subset of a segment. In an EPG representation a set of sub-segments (eventually from different segments) is called a \emph{part}. 

Furthermore,  a vertex $u$ is \emph{displayed}, if there is at least one grid edge, which is exclusively in $P(u)$ and in no other path. We then say that this grid edge is a \emph{private} edge of $P(u)$. An edge $\{u,v\}\in E$ is \emph{displayed}, if there is at least one grid edge in $P(u) \cap P(v)$, which is only contained in $P(u)\cap P(v)$ and not element of any other path. This grid edge is then called a \emph{private} edge of $P(u)\cap P(v)$.
We also say that a part (or sub-segment) \emph{displays} the corresponding vertex or edge if it consists only of grid edges which are private edges of the respective vertex or edge.

Finally, two horizontal sub-segments in an EPG representation \emph{see each other} if there is a vertical grid line crossing both sub-segments. Similarly, two vertical sub-segments see each other if there is a horizontal grid line crossing both.

\section{EPG representations of graphs in terms of treewidth}
In this section we consider graphs of bounded treewidth. Therefore we denote the treewidth of a graph $G$ with $tw(G)$ and make use of the fact, that every graph $G$ with $tw(G)\leq k$ is a subgraph of a $k$-tree as well as the fact, that every $k$-tree admits a construction sequence, starting with a $(k+1)$-clique and iteratively stacking a new vertices into $k$--cliques. For further definitions and more about construction sequences of graphs of bounded treewidth we refer to~\cite{Bod-98}.
\begin{theorem}\label{bendThm:tw2-up}
 For every graph $G$ with $tw(G) \leq 2$ we have $b(G) \leq 2$.
\end{theorem}
\begin{proof}
 Let $\tilde{G}$ be the $2$-tree which contains $G$ and $(v_1,\ldots,v_n)$ be a vertex ordering of $G$ implied by $\tilde{G}$'s construction sequence. We construct a $2$-bend representation of $G$ along the building sequence $G_2 \subset \ldots \subset G_n = G$, where we add vertex $v_i$ to $G_i$, such that the two neighbors of $v_i$ in $\tilde{G}_i$ form a $2$-clique in $\tilde{G}_i$ (and not necessarily in $G_i$). We maintain that $\Gamma_i$ is a $2$-bend representation of $G_i$, such that every $2$-clique in $\tilde{G}_i$ satisfies one of the two invariants in Fig.~\ref{bendFig:tw2-up}~\textbf{a)}, i.e., for $i = 2,\ldots,n$ and $\{u,v\} \in E(\tilde{G}_i)$ we have a sub-segment $p_u \subseteq P(u)$ and a part $p_v \subseteq P(v)$ such that one of the following sets of conditions holds:

 \begin{enumerate}[label=(\roman*)]
\item $p_v$ is a sub-segment, $p_u$ and $p_v$ see each other, $p_u$ displays $u$, and $p_v$ displays $v$ as depicted in the top row of Fig.~\ref{bendFig:tw2-up}~\textbf{a)},

 \item $p_v$ consists of two consecutive sub-segments $p_{v,1}, p_{v,2}$ such that there is a bend between them, 
 $p_{v,2}$ displays $v$,  
 $p_u \setminus p_{v,1}$ displays $u$,  and
 $p_u \cap p_{v,1}$ displays $\{u,v\}$ as depicted in the bottom row of Fig.~\ref{bendFig:tw2-up}~\textbf{a)}.
 \end{enumerate}

 \begin{figure}[htb]
  \centering
  \psfrag{1}[bc][bl]{$u$}
  \psfrag{2}[bc][bl]{$v$}
  \psfrag{a}[bc][bl]{\textbf{a)}}
  \psfrag{b}[bc][bl]{\textbf{b)}}
  \psfrag{c}[bc][bl]{\textbf{c)}}
  \psfrag{d}[bc][bl]{\textbf{d)}}
  \psfrag{x}[bc][bl]{$v_i$}
  \psfrag{>}[bc][b]{$\longrightarrow$}
  \psfrag{<}[bc][b]{$\longleftarrow$}
  \includegraphics{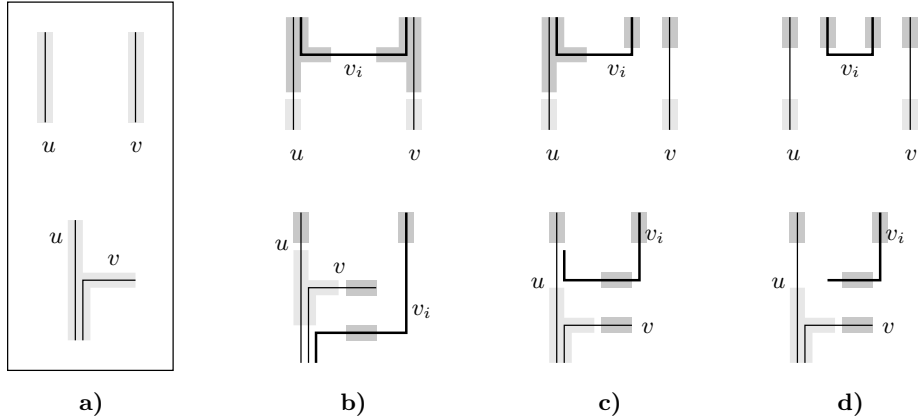}
  \caption{Invariants for a $2$-clique and insertion rules for the path for the new vertex $v_i$ (drawn bold). \textbf{a)} $p_u$ and $p_v$ for the two types of the invariant. \textbf{b)} Vertex $v_i$ has an edge with both, $u$ and $v$. \textbf{c)}: Vertex $v_i$ has an edge with $u$ and no edge with $v$. \textbf{d)}: Vertex $v_i$ has no edge with $u$ or $v$.}
  \label{bendFig:tw2-up}
 \end{figure}

 The two possible starting representations for $G_2$ are shown in Fig.~\ref{bendFig:tw2-up}~\textbf{a)}.
 Now for $i \geq 2$, let $\Gamma_i$ be a $2$-bend representation of $G_i$ that satisfies our invariant.
 
 Let $\{u,v\}$ be the $2$-clique that is the neighborhood of $v_i$ in $\tilde{G}_i$. If $v_i$ has an edge in $G_{i+1}$ with both, $u$ and $v$, we introduce the path for $v_i$ as illustrated in Fig.~\ref{bendFig:tw2-up}~\textbf{b)} depending on the type of the invariant for $\{u,v\}$. The parts, displaying the new $2$-cliques $\{u,v_i\}$ and $\{v,v_i\}$ in $\tilde{G}_{i+1}$, are highlighted in dark gray and the part, which displays  $\{u,v\}$, is highlighted in light gray.

 If $v_i$ has an edge in $G_{i+1}$ with $u$ and no edge in $G_{i+1}$ with $v$, we introduce the path for $v_i$ as illustrated in Fig.~\ref{bendFig:tw2-up}~\textbf{c)} depending on the type of invariant we have. Again the parts, which display the new $2$-cliques $\{u,v_i\}$ and $\{v,v_i\}$, and the old $2$-clique $\{u,v\}$ in $\tilde{G}_{i+1}$ are highlighted in dark gray, and light gray, respectively. If $v_i$ has an edge in $G_{i+1}$ with $v$ and no edge in $G_{i+1}$ with $u$, the roles of $u$ and $v$ are basically exchanged. In case, $v_i$ neither has an edge in $G_{i+1}$ with $u$ nor with $v$, we introduce the path for $v_i$ as illustrated in Fig.~\ref{bendFig:tw2-up}~\textbf{d)} depending on the type of type of invariant for $\{u,v\}$.
\end{proof}

Since every outerplanar graph $G$ has a treewidth $\leq 2$ we obtain the following corollary.
\begin{corollary}
 For every outerplanar graph $G$ we have $b(G) \leq 2$.
\end{corollary}
This confirms a conjecture of Biedl and Stern~\cite{Bie-10}, who show that the graph in  Fig.~\ref{fig:out-tight} has no 1-bend representation. Therefore this bound cannot be further improved.


\begin{theorem}\label{bendThm:stw3-up}
 For every planar graph $G$ with $tw(G) \leq 3$ we have $b(G) \leq 3$.
\end{theorem}
\begin{proof}
By a result of El-Mallah and Colbourn~\cite{Elm-90} $G$ is a subgraph of a plane $3$-tree $\tilde{G}$. So there is a vertex ordering $(v_1,\ldots,v_n)$, such that $\tilde{G}_3$ is a triangle $\{v_1, v_2, v_3\}$ and $\tilde{G}_i$ is obtained from $\tilde{G}_{i-1}$ by connecting $v_i$ to the three vertices $u,v,w$ of a triangle bounding an inner face of $\tilde{G}_{i-1}$. The triangle $\{u,v,w\}$ is then not bounding an inner face of $\tilde{G}_i$ anymore and hence no second vertex may be attached to it. We build a $3$-bend representation of $G$ concurrently with the building sequence $G_3 \subset \cdots \subset G_n$ of $G$ w.r.t. the vertex ordering $(v_1,\ldots,v_n)$. We maintain the following invariant on the $3$-bend representation $\Gamma_i$ of $G_i$, for $i \geq 3$:

 \begin{enumerate}[label=(\alph*)]
  \item Every vertex $u\in \Gamma_i$ has a  horizontal \emph{and} a vertical sub-segment displaying $u$, and\label{bendEnum:displayed}
  \item every facial triangle $\{u,v,w\}$ of $\tilde{G}_i$ has two vertices, say $u$ and $v$, of\label{bendEnum:triangle} one of the following two types:
   \begin{enumerate}[label=(\roman*)]
    \setlength{\itemsep}{1pt}
    \setlength{\parskip}{0pt}
    \setlength{\parsep}{0pt}
    \item  there is  a sub-segment, which displays the edge $\{u,v\}$ as in the left example in Fig.~\ref{bendFig:stw3-up}~\textbf{a)},
    \item  there is an \emph{entire} segment $s$ of $P(v)$ displaying $v$ and a sub-segment of $P(u)$ displaying $u$ and crossing  $s$, see Fig.~\ref{bendFig:stw3-up}~\textbf{a)}.\label{bendEnum:cross}
   \end{enumerate}
 \end{enumerate}
 
 Moreover, we require that all the displaying parts above are pairwise disjoint and that the entire segment in every cross in~\ref{bendEnum:cross} \emph{cannot} see displaying parts from~\ref{bendEnum:displayed}. However, we need this assumption only in the case in Fig.~\ref{bendFig:stw3-up}~\textbf{c)}.

 \begin{figure}[htb]
  \centering
  \psfrag{1}[bc][bl]{$u$}
  \psfrag{2}[bc][bl]{$v$}
  \psfrag{3}[bc][bl]{$w$}
  \psfrag{x}[bc][bl]{$v_i$}
  \psfrag{A}[bc][bl]{\textbf{a)} invariants for $\{u,v\}$}
\psfrag{A2}[bc][bl]{of type (i) and (ii)}
  \psfrag{B}[bc][bl]{\textbf{b)} $\deg_{G_i}(v_i) = 3$,}
  \psfrag{B2}[bc][bl]{$\{u,v\}$ of first type}
  \psfrag{C}[bc][bl]{\textbf{c)} $\deg_{G_i}(v_i) = 3$,}
  \psfrag{C2}[bc][bl]{$\{u,v\}$ of second type}
  \psfrag{D}[bc][bl]{\textbf{d)} $\deg_{G_i}(v_i) = 2$}
  \psfrag{E}[bc][bl]{\textbf{e)} $\deg_{G_i}(v_i) = 1$}
  \psfrag{F}[bc][bl]{\textbf{f)} $\deg_{G_i}(v_i) = 0$}
  \psfrag{or}[bc][bl]{or}
  \includegraphics{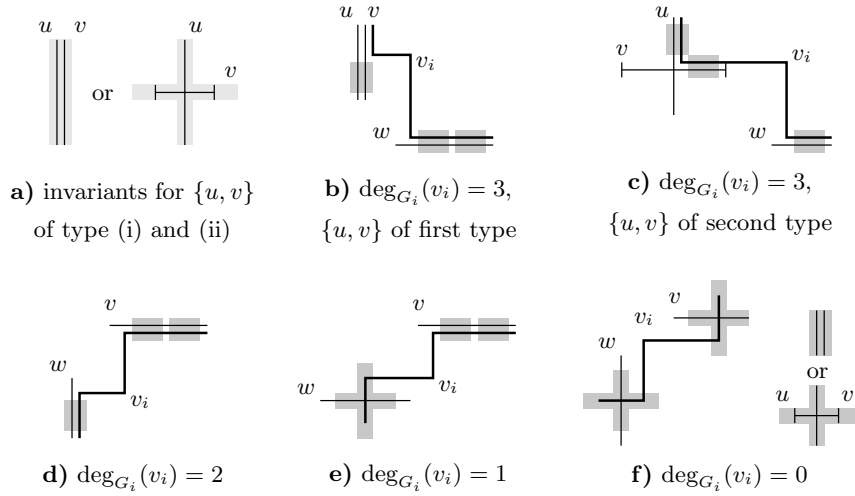}
  \caption{Building a $3$-bend representation of a planar graph with tree-width $3$, a vertex $v_i$ is attached to the facial triangle $\{u,v,w\}$ in $\tilde{G}_{i-1}$ : In \textbf{a)} the two types of invariant for $\{u,v\}$ are shown. In \textbf{b)}--\textbf{f)} it is shown how to insert the new vertex $v_i$ (drawn bold) depending on its degree in $G_i$ and the invariant of $\{u,v\}$. The invariants for the three new facial triangles $\{u,v,v_i\}$, $\{u,v_i,w\}$, and $\{v_i,v,w\}$ in $\tilde{G}_i$ are highlighted.}
  \label{bendFig:stw3-up}
 \end{figure}

 It is not difficult to find a $3$-bend representation $\Gamma_3$ of the subgraph $G_3$ of $G$, which satisfies invariant~\ref{bendEnum:displayed} and~\ref{bendEnum:triangle}. Indeed, by introducing three artificial outer vertices we can assume w.l.o.g. that $G_3$ is an independent set of size three. 
%
%

 For $i\geq 4$, the path for vertex $v_i$ is introduced to $\Gamma_{i-1}$ according to the degree of $v_i$ in $G_i$ and the type of invariant for the facial triangle $\{u,v,w\}$ in $\tilde{G}_{i-1}$ that $v_i$ is connected to. Fig.~\ref{bendFig:stw3-up}~\textbf{b)}--\textbf{f)} shows all five cases and how to introduce $v_i$, which is illustrated by the bold path. Consider in particular the case that $v_i$ has an edge with each of $u$, $v$, and $w$, and moreover the triangle $\{u,v,w\}$ has a cross, see Fig.~\ref{bendFig:stw3-up}~\textbf{c)}. Here we use that the entire $v$-segment cannot see the partial $w$-segment in order to get a horizontal displaying sub-segment for the new path corresponding to $v_i$.

 In the figure the parts from invariant~\ref{bendEnum:triangle} displaying edges of  the new facial triangles $\{u,v,v_i\}$, $\{u,v_i,w\}$ and $\{v_i,v,w\}$ in $\tilde{G}_i$ are highlighted in dark gray. Additionally, every path, including the new path for $v_i$, has a horizontal and a vertical sub-segment, which displays them. Moreover they can be chosen such that they do not see any entire segment of a cross.

 For example, consider the case that $v_i$ is not adjacent to $u$, $v$, or $w$ in $G_i$, see Fig.~\ref{bendFig:stw3-up}~\textbf{f)}. Here the new facial triangles $\{u,v_i,w\}$ and $\{v_i,v,w\}$ have a cross of an entire segment of $P(v_i)$ and a sub-segment displaying $w$ and $v$, respectively. The new facial triangle $\{u,v,v_i\}$ has the edge $\{u,v\}$, which still satisfies both parts of the invariant.
\end{proof}

We conclude this section showing that the bound in Theorem~\ref{bendThm:stw3-up} is tight. This confirms what Biedl and Stern strongly suspected in~\cite{Bie-10}.

\begin{proposition}\label{bendLem:stw3-low}
 There is a planar graph $G$ with $tw(G) = 3$ and $b(G) = 3$.
\end{proposition}
The graph $G$ is depicted in the left of Fig.~\ref{fig:stw3-low} and it is constructed in the following way: Two vertices $u$ and $v$ together with $50$ white vertices form an induced $K_{2,50}$ subgraph, i.e., a complete bipartite graph on $2 + 50$ vertices. Any two consecutive white vertices together with $50$ black vertices form another induced $K_{2,50}$. Finally, between any two consecutive black vertices the graph $H$ is suspended, which is the $29$-vertex graph depicted in the right of Fig.~\ref{fig:stw3-low}. This graph is clearly planar. Furthermore one can find a planar $3$-tree containing $G$, i.e., $G$ is planar and has treewidth at most $3$.
\begin{figure}[htb]
 \centering
 \psfrag{G=}[bl][bl]{$G=$}
 \psfrag{H=}[bl][bl]{$H=$}
 \psfrag{H}[bc][bl]{$H$}
 \psfrag{x}[bc][bl]{$x$}
 \psfrag{y}[tc][tl]{$y$}
 \psfrag{u}[bc][bl]{$u$}
 \psfrag{v}[bc][bl]{$v$}
 \psfrag{m1}[tc][tl]{$50$}
 \psfrag{m2}[cr][bl]{$50$}
 \psfrag{.}[cc][bl]{$\vdots$}
 \psfrag{..}[cc][bl]{$\cdots$}
 \psfrag{1}[br][bl]{$a_1$}
 \psfrag{2}[bc][bl]{$a_2$}
 \psfrag{3}[bc][bl]{$a_3$}
 \psfrag{4}[bl][bl]{$a_4$}
 \psfrag{g}[bc][bl]{$g$}
 \psfrag{w2}[br][bl]{$w_2$}
 \psfrag{w3}[bl][bl]{$w_3$}
 \includegraphics{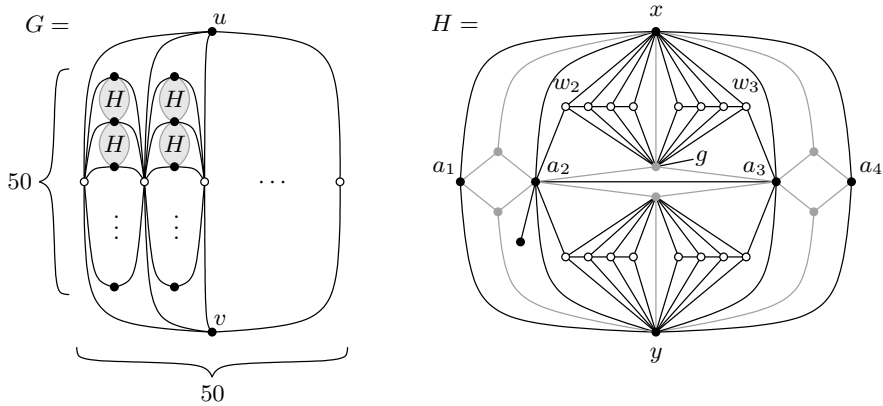}
 \caption{A planar graph $G$ with $tw(G) = 3$ and $b(G) = 3$.}
 \label{fig:stw3-low}
\end{figure}

Now we will show that $2$ bends do not suffice for an EPG representation of $G$.
\begin{proof}
 It is known~\cite{Asi-09} that $b(K_{2,n}) = 2$ if $n \geq 5$. Hence $b(G) \geq 2$. For the sake of contradiction let us assume that $b(G) = 2$ and consider a $2$-bend representation of $G$ together with the induced $2$-bend representations of all induced $K_{2,50}$ subgraphs. Furthermore, it is known~\cite{Bie-10} that in any EPG representation of $K_{2,n}$ at most $12$ vertices of the $n$-bipartition class represent their two edges using the same or consecutive segments. Moreover the $2$-bend paths of $u$ and $v$ together have at most $12$ endpoints of segments. Thus at most another $12$ white vertices contain an endpoint of $u$ or $v$ in the interior of a segment. Hence there are two consecutive white vertices, say $u'$ and $v'$, that establish their edges with $u$ and $v$ with two distinct segments of the same direction, say vertically. Both these segments are completely contained in the corresponding segment of $u$ and $v$.

 Now the same argument can be applied to the $K_{2,50}$ subgraph induced by $u'$, $v'$, and $50$ black vertices. Thus there are two consecutive black vertices, which we denote by $x$ and $y$, that establish their two edges with $u'$ and $v'$ with two distinct segments of the same direction, each completely contained in the corresponding segment of $u'$ and $v'$. Moreover, $x$ and $y$ intersect $u'$ and $v'$ on \emph{the same segment} of $u'$ and $v'$, respectively. We conclude that the paths of $u'$, $v'$, $x$, and $y$ are positioned as the black bold paths in Fig.~\ref{fig:stw3-low-detail}.

 \begin{figure}[htb]
  \centering
  \psfrag{u}[br][bl]{$u'$}
  \psfrag{v}[bl][bl]{$v'$}
  \psfrag{x}[bl][bl]{$x$}
  \psfrag{y}[tl][tl]{$y$}
  \psfrag{g}[bl][bl]{$g$}
  \psfrag{1}[bl][bl]{$a_1$}
  \psfrag{2}[bl][bl]{$a_2$}
  \psfrag{3}[bl][bl]{$a_3$}
  \psfrag{4}[bl][bl]{$a_4$}
  \psfrag{l}[bl][bl]{$\ell$}
  \includegraphics{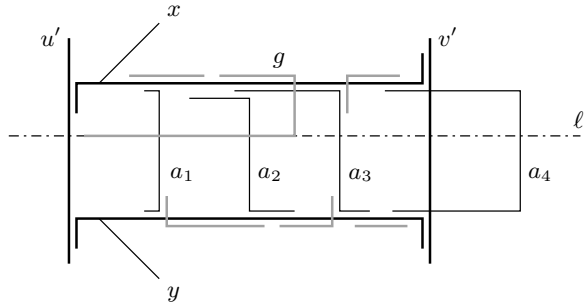}
  \caption{A detail of a hypothetical $2$-bend representation of $G$.}
  \label{fig:stw3-low-detail}
 \end{figure}

 Now consider the subgraph $H$ which is suspended between $x$ and $y$. W.l.o.g. we assume that the segments which display $x$ and $y$ are horizontal. The graph $H$ contains, aside from $x$ and $y$, a set $A$ of four black vertices denoted by $a_1$, $a_2$, $a_3$, and $a_4$, and a set $B$ of six gray vertices, see the right of Fig.~\ref{fig:stw3-low}. Every vertex in $A \cup B$ is adjacent to $x$ or $y$, but neither to $u'$ nor to $v'$. Hence each edge of such a vertex with $x$ or $y$ is established horizontally. From the adjacencies between vertices in $A \cup B$ it follows that the vertical segments of $a_1$, $a_2$, $a_3$, and $a_4$ appear in order -- w.l.o.g. from left to right. Moreover the edge $\{a_2,a_3\}$ is not established vertically since $a_2$ has a private neighbor. W.l.o.g. assume that $\{a_2,a_3\}$ is established within the $x$-segment, see Fig.~\ref{fig:stw3-low-detail}. Let us denote the gray vertex adjacent to $a_2$, $a_3$, and $x$ by $g$.

 It follows that the horizontal segment of $g$ within the $x$-segment is completely covered by $a_2$ and $a_3$ and hence lies strictly between $a_1$ and $a_4$. In consequence, the six white vertices that are adjacent to $x$ and $g$, but not to $a_2$ or $a_3$, intersect $g$ on its \emph{second horizontal} segment. Let us denote the set of eight white vertices, which are common neighbors of $g$ and $x$ by $W$ and the horizontal line supporting at least six edges between $g$ and $W$ by $\ell$. Evidently, the segment of $g$ on $\ell$ crosses only one of $\{a_1,a_4\}$, but not both, say it does \emph{not} cross $a_4$, see Fig.~\ref{fig:stw3-low-detail}.

 The vertex in $W$ that is adjacent to $a_2$, respectively $a_3$, is denoted by $w_2$, respectively $w_3$. One can check that for $i = 1,2$ the edge between $w_i$ and its white neighbor is established on $\ell$ and moreover that $w_2$ lies to the left of $w_3$ on $\ell$. This implies that the white vertices inside the triangle $\{x,g,w_3\}$ intersect the horizontal $x$-segment to the right of $a_4$. In order to establish their adjacency with $g$ all these three paths contain the part of $\ell$ between $a_3$ and $a_4$, making them pairwise adjacent, which in $G$ they are not -- a contradiction.
\end{proof}

\section{$4$-bend representation for planar graphs}

We show that the bend-number of every planar graph is at most $4$, improving the recent upper bound of $5$ due to Biedl and Stern~\cite{Bie-10}. Our proof is constructive and can indeed be seen as a linear-time algorithm to find a $4$-bend representation of any given planar graph. We use the folklore fact that every plane triangulation can be constructed from a triangle by successively glueing \hbox{$4$-connected} triangulations into inner faces. In every step our algorithm constructs a $4$-bend representation $\Gamma'$ of a $4$-connected triangulation $G'$ and incorporates it into the already defined representation. The construction of $\Gamma'$ is based on a well-known representation of subgraphs of 4-connected triangulations by touching axis-aligned rectangles. These are the basic steps of the algorithm:

 \begin{enumerate}[label=\arabic*.)]
  \item Fix some plane embedding of the given planar graph and add one vertex into each face to obtain a super-graph $G$ that is a triangulation. If we find a $4$-bend representation for $G$, removing the paths from it that correspond to added vertices, results in a $4$-bend representation of the original graph.
  \item Construct a $4$-bend representation of the outer triangle of $G$ so that invariant $I$, presented ahead, is satisfied.
  \item Let $\Gamma'$ and $G'$ denote the so far defined $4$-bend representation and the graph that is represented, respectively. The graph $G'$ will always be a plane triangulation, which is an induced subgraph of $G$. We repeat the following two steps 4.) and 5.) until we end up with a $4$-bend representation $\Gamma$ of the entire triangulation $G$.
  \item Consider a triangle $\Delta$ which is an inner face of $G'$, but not of $G$. Let $G_\Delta$ be the unique $4$-connected triangulated subgraph of $G$, which contains $\Delta$ and at least one vertex lying inside of it. (No vertex in $G_\Delta$, except those in $\Delta$, is represented in $\Gamma'$.)
  \item Construct a $4$-bend path for every vertex in $G_\Delta \backslash \Delta$ and add it to the representation $\Gamma'$ so that all edges of $G_\Delta$ are properly represented, i.e., $\Gamma'$ is a $4$-bend representation of $G' \cup G_\Delta$. Additionally, we ensure that our invariant is satisfied for every inner facial triangle of $G' \cup G_\Delta$.
 \end{enumerate}

The algorithm described above computes a $4$-bend representation for every given planar graph. Moreover, steps 1.) to 3.) can easily be executed in linear time. For an efficient implementation of step~4.) we first identify all separating triangles of $G$, which can be in linear time~\cite{ItaRod-78}. The graph $G_\Delta$ then arises from $G$ by removing all vertices in the exterior of the triangle $\Delta$ and all vertices in the interior of every remaining separating triangle different from $\Delta$. This can for example be done in $\mathcal{O}(|V(G_\Delta)|)$ if $G$ is stored with adjacency lists.

The crucial part is step 5.), i.e., the construction of a $4$-bend representation of the \hbox{$4$-connected} triangulation $G_\Delta$. Our construction relies on a well-known geometric representation of proper subgraphs of $4$-connected plane triangulations. Actually, we need the slight strengthening in Lemma~\ref{lem:box-representation} below. A \emph{plane non-separating near-triangulation}, \emph{NST} for short, is a plane graph on at least four vertices without separating triangles such that all inner faces are triangles and the outer face is a quadrangle. For example, if $G$ is a plane $4$-connected triangulation and $e = (u,w)$ an outer edge of $G$, then $G \backslash (u,w)$ is an NST.

\begin{lemma}\label{lem:box-representation}
 Let $G = (V,E)$ be an NST. Then $G$ can be represented as follows:
 \begin{enumerate}[leftmargin=3em,label=(\alph*)]
  \item There is an axis-aligned rectangle $R(v)$ for every $v \in V$.\label{enum:box-one}
  \item Any two rectangles are either disjoint or intersect on a line segment of non-zero length.\label{enum:box-two}
  \item The rectangles $R(v)$ and $R(w)$ have a non-empty intersection if and only if $(v,w)$ is an edge in $G$.\label{enum:box-three}
 \end{enumerate}
 Additionally, for every $v \in V$ there is a vertical segment $t_v$ from the bottom to the top side of $R(v)$, such that the following holds:
 \begin{enumerate}[resume,leftmargin=3em,label=\textbf{(\alph*)}]
  \item The segment $t_v$ lies to the right of $t_w$ for every $v \in V$ and $R(w)$ touching the top side of $R(v)$.\label{enum:box-five}
 \end{enumerate}
\end{lemma}
\begin{proof}
 It is well-known that every NST has a representation with touching rectangles, i.e., items~\ref{enum:box-one}--\ref{enum:box-three} hold~\cite{KozKin-85,LaiLei-84,Ung-53}. Moreover, those rectangles can be computed in linear time. We take such a rectangle representation and apply two modifications to it as illustrated in the top and bottom of Fig.~\ref{fig:pla-up-distort}, respectively. The modifications operate on the incidence relation of segments in the representation, that is, the inclusion-maximal horizontal and vertical lines. However, what follows is a more easily interpreted graphic image of the situation.

 For the first modification we define for every \emph{inner} rectangle $R(v)$ a line $\ell$, like the dashed one in Fig.~\ref{fig:pla-up-distort}. To be precise, $\ell$ consists of two vertical rays, one starting at the contact between $R(v)$ and the leftmost rectangle touching the bottom side of $R(v)$, and one starting at the contact between $R(v)$ and the rightmost rectangle touching the top side of $R(v)$. Those vertical rays are connected inside $R(v)$, e.g., by a straight line segment. Then the plane is cut along $\ell$ and pulled apart to the left and right while extending every rectangle that has been cut by $\ell$. As soon as the lower ray in the right part lies to the right of the upper ray in the left part we define the segment $t_v$ to lie between those two rays.

Having applied this modification to every inner rectangle, each segment $t_v$ is spanned between the rightmost rectangle touching the top of $R(v)$ and the leftmost rectangle touching the bottom of $R(v)$, for every $v \in V$.

 \begin{figure}[htb]
  \psfrag{>}[bc][bl]{$\leadsto$}
  \psfrag{x}[cc][bc]{$R(v)$}
  \psfrag{y}[cc][bc]{$R(w)$}
  \psfrag{t}[cc][bc]{$t_v$}
  \psfrag{l}[cc][bc]{$\ell$}
  \psfrag{l'}[cc][bc]{$\ell'$}
  \centering
  \includegraphics{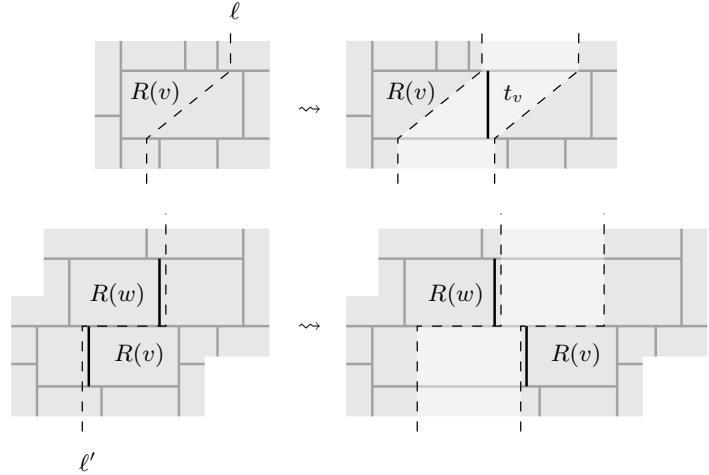}
  \caption{\textbf{Top:} The modification at an inner rectangle and the positioning of the segment $t_v$. \textbf{Bottom:} The modification at a pair $R(v)$, $R(w)$, with $R(w)$ being the rightmost rectangle touching the top of $R(v)$.}
  \label{fig:pla-up-distort}
 \end{figure}

 For the second transformation another line $\ell'$ is defined for every pair of a rectangle $R(v)$ and its rightmost rectangle $R(w)$ touching the top side of $R(v)$. Again $\ell'$ contains two rays which both start in $R(v) \cap R(w)$. The ray emerging downwards starts to the left of $t_v$ and the ray going upwards starts to the right of $t_w$. Both rays are connected by a straight line segment contained in $R(v) \cap R(w)$. Similarly to the first modification, the plane is cut along $\ell'$ and pulled apart (extending every rectangle that is cut by $\ell'$) until $t_w$ lies to the left of $t_v$.

 Having applied this second modification to every pair of a rectangle $R(v)$ and its right-most top neighbor $R(w)$, condition~\ref{enum:box-five} in Lemma~\ref{lem:box-representation} is fulfilled. Moreover, operating on the combinatorics of segments only, each modification can be done in constant time. Hence we obtain a linear-time algorithm, which completes the proof.
\end{proof}

\begin{remark}
 There is an alternative 
proof of Lemma~\ref{lem:box-representation} that, however, involves some concepts. It is known that rectangle representations are in bijection with so-called \emph{transversal structures}~\cite{Fus-07}. Given a transversal structure, one can split a vertex $v$ into two adjacent vertices $v'$ and $v''$, distribute the neighbors of $v$ among $v'$ and $v''$ to have a triangulation again, and carry over the transversal structure to the new graph. And all this can be done in such a way that in the rectangle representation corresponding to the new transversal structure the right side of $R(v')$ coincides with the left side of $R(v'')$. In other words, $R(v') \cup R(v'')$ is a rectangle with a vertical segment $t_v$. 
This attempt directly gives a linear-time algorithm to compute a rectangle representation satisfying items~\ref{enum:box-one}--\ref{enum:box-five} in Lemma~\ref{lem:box-representation}.
\end{remark}

We will refer to a set of rectangles satisfying~\ref{enum:box-one}--\ref{enum:box-five} in Lemma~\ref{lem:box-representation} as a \emph{rectangle representation of $G$}. Rectangle representations are also known as \emph{rectangular duals}. It is easily seen that every inner (triangular) face $\Delta = \{u,v,w\}$ of $G$ corresponds to a unique point in the plane given by $R(u) \cap R(v) \cap R(w)$, i.e., the common intersection of the three corresponding rectangles.

\bigskip

Now we describe our invariant mentioned in steps~2.) and~5.) above.
\subsubsection{Invariant $I$:}
Let $G'$ be a plane triangulation. A  $4$-bend representation $\Gamma'$ of $G'$ is said to \emph{satisfy invariant~$I$} if there are mutually disjoint regions, i.e., polygonal parts of the plane, associated with the inner facial triangles of $G'$.
For each inner facial triangle $\Delta=\{u,v,w\}$ we require that inside \emph{the region for $\Delta$} there are segments $s_u$, $s_v$, $s_w$, and $s_{uv}$ displaying $u,v,w$, and the edge $\{u,v\}$, respectively, such that (up to grid-symmetry) each of the following holds. 
\begin{enumerate}[leftmargin=5em,label={(\alph*)}]
 \item Segment $s_u$ is vertical, $s_v$ is horizontal, lying to the top-right of $s_u$.
 \item Segment $s_{uv}$ lies above $s_u$ on the same grid line.
 \item Segment $s_w$ lies to the right of $s_u$ and not to the left of $s_v$, and is of one of the following types
  \begin{enumerate}[leftmargin=5em,label={(\roman*)}]
   \item $s_w$ is vertical and sees $s_u$.
   \item $s_w$ is horizontal and sees $s_v$.
  \end{enumerate}
\end{enumerate}

\begin{figure}[htb]
 \psfrag{A}[bc][bl]{type (i)}
 \psfrag{B}[bc][bl]{type (ii)}
 \psfrag{u}[bc][bl]{$s_u$}
 \psfrag{v}[bc][bl]{$s_v$}
 \psfrag{w}[bc][bl]{$s_w$}
 \psfrag{o}[bc][bl]{or}
 \psfrag{uv}[bc][bl]{$s_{uv}$}
 \centering
 \includegraphics{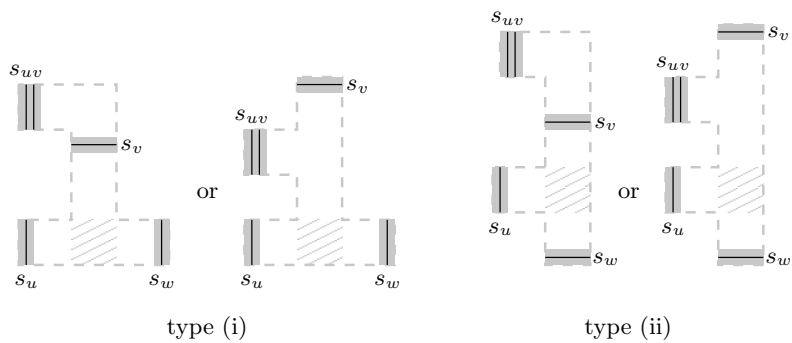}
 \caption{The regions of type (i) and type (ii) for a triangle $\Delta = \{u,v,w\}$ (up to grid-symmetry). The part of the region that can be seen by each of $s_u$, $s_v$, and $s_w$ is highlighted.}
 \label{fig:planar-invariant}
\end{figure}

See Fig.~\ref{fig:planar-invariant} for an illustration. The regions are illustrated by the dashed line in Fig.~\ref{fig:planar-invariant}. It is important that the regions in Fig.~\ref{fig:planar-invariant} may appear reflected and/or rotated by $90$ degree.

\bigskip

Our algorithm to find a $4$-bend representation of a planar graph executes steps 1.) to 5.) as described above. It is not difficult to find a $4$-bend representation $\Gamma'$ of a triangle in step 2.) which satisfies invariant $I$, see Fig.~\ref{fig:planar-base}.

\begin{figure}[htb]
\psfrag{u}[bc][bl]{$s_u$}
\psfrag{v}[bc][bl]{$s_v$}
\psfrag{w}[bc][bl]{$s_w$}
\psfrag{uv}[bc][bl]{$s_{uv}$}
\centering
\includegraphics{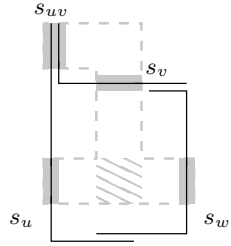}
\caption{An EPG representation of a triangle $\Delta = \{u,v,w\}$ satisfying invariant $I$ of type (i).}
\label{fig:planar-base}
\end{figure}

Hence we only have to take care of step 5.), which boils down to the following.

\begin{lemma}\label{lem:planar-step5}
 Let $\{P(u),P(v),P(w)\}$ be a $4$-bend representation of a triangle $\Delta = \{u,v,w\}$ which satisfies invariant~$I$, and $G$ be a $4$-connected triangulation whose outer triangle is $\Delta$. Then $\{P(u),P(v),P(w)\}$ can be extended to a $4$-bend representation $\Gamma$ of $G$ which satisfies invariant~$I$, such that every new $4$-bend path, as well as the region for every inner face of $G$, lies inside the region for $\Delta$. Moreover, such a representation can be found in linear time.
\end{lemma}
\begin{proof}
 For convenience we rotate and/or flip the entire representation such that the region for $\Delta$ is similar to one of Fig.~\ref{fig:planar-invariant}, i.e., $s_u$ and $s_{uv}$ are the leftmost segments and $s_{uv}$ lies above $s_u$. Define $G' = G \setminus \{u,w\}$ if the region for $\Delta$ is of type~(i), and $G' = G \setminus \{v,w\}$ if the region for $\Delta$ is of type~(ii). By Lemma~\ref{lem:box-representation} $G'$ can be represented by rectangles so that conditions~\ref{enum:box-one}--\ref{enum:box-five} are met. We put this rectangle representation into the part of the region for $\Delta$ which can be seen by each of $s_u$, $s_v$, and $s_w$. This part is highlighted in Fig.~\ref{fig:planar-invariant}.

 Now we explain how to replace the rectangle by a path for every vertex $i \neq u,v,w$ in $G'$. We start with a snake-like $4$-bend path $P(i)$ within the rectangle corresponding to~$i$. To be precise, $P(i)$ starts at the bottom-left corner of $R(i)$, goes up to the top-left corner, where it bends and goes right up to the segment $t_i$, which is defined in Lemma~\ref{lem:box-representation}. Then $P(i)$ follows along $t_i$ down to the bottom side of $R(i)$, where it bends and goes to the right up to the bottom-right corner and then up to the top-right corner. Afterwards, every path $P(i)$ is shortened at both ends by some amount small enough that no edge-intersection is lost. See Fig.~\ref{fig:pla-up-snake} for an illustration. Note that if $R(j)$ touches the top side of $R(i)$, then the paths $P(i)$ and $P(j)$ do have an edge-intersection because $t_i$ lies to the right of $t_j$.
 
 \begin{figure}[htb]
  \psfrag{>}[bc][bl]{$\leadsto$}
  \psfrag{ti}[cc][bc]{$t_i$}
  \psfrag{Pi}[cc][bc]{$P(i)$}
  \centering
  \includegraphics{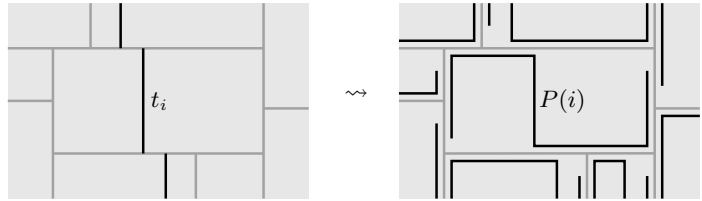}
  \caption{The $4$-bend snake-like path $P(i)$ within the rectangle $R(i)$ corresponding to $i$.}
  \label{fig:pla-up-snake}
 \end{figure}

 It is not difficult to see, that these snake-like paths form a $4$-bend representation of $G' \backslash \{u,v,w\}$. The remaining edges in $G'$, i.e., those incident to $u$, $v$ or $w$, are established by extending the paths corresponding to the neighbors of $u$, $v$ and $w$ in $G'$. What follows is illustrated with an example in Fig.~\ref{fig:planar-insert}. Let $x$ be the unique\footnote{There is no second such vertex, since $G'$ has no separating triangles.} neighbor of $u$ and $v$ in $G$. For every neighbor $i\neq x$ of $v$ the upper horizontal segment of $P(i)$ is shifted up onto the segment $s_v$ (extending the left and the middle vertical segment of $P(i)$). Similarly, for every neighbor $i\neq x$ of $u$ the leftmost vertical segment of $P(i)$ is shifted to the left onto the segment $s_u$ (extending the upper horizontal segment of $P(i)$). For the case $i=x$, the upper horizontal segment of $P(x)$ is shifted up onto some horizontal line through $s_{uv}$. Then the leftmost vertical segment of $P(x)$ is shifted to the left onto the segment $s_{uv}$, and shortened such that it is completely contained in $s_{uv}$. The paths of neighbors of $w$ are extended according to the type of the region for $\Delta$. If the region is of type~(i), the rightmost vertical segment of $P(i)$ for every neighbor $i$ of $w$ is shifted to the right onto $s_w$. If the region is of type~(ii), the lower horizontal segment of $P(i)$ is shifted down onto $s_w$. Let us again refer to Fig.~\ref{fig:planar-insert} for an illustration.

 \begin{figure}[htb]
  \psfrag{u}[cc][bl]{$s_u$}
  \psfrag{v}[cc][bl]{$s_v$}
  \psfrag{w}[cc][bl]{$s_w$}
  \psfrag{x}[cc][bl]{$P(x)$}
  \psfrag{y}[cc][bl]{$P(y)$}
  \psfrag{z}[cc][bl]{$P(z)$}
  \psfrag{uv}[cc][bl]{$s_{uv}$}
  \psfrag{A}[cc][bl]{\textbf{type (i)}}
  \psfrag{B}[cc][bl]{\textbf{type (ii)}}
  \centering
  \includegraphics{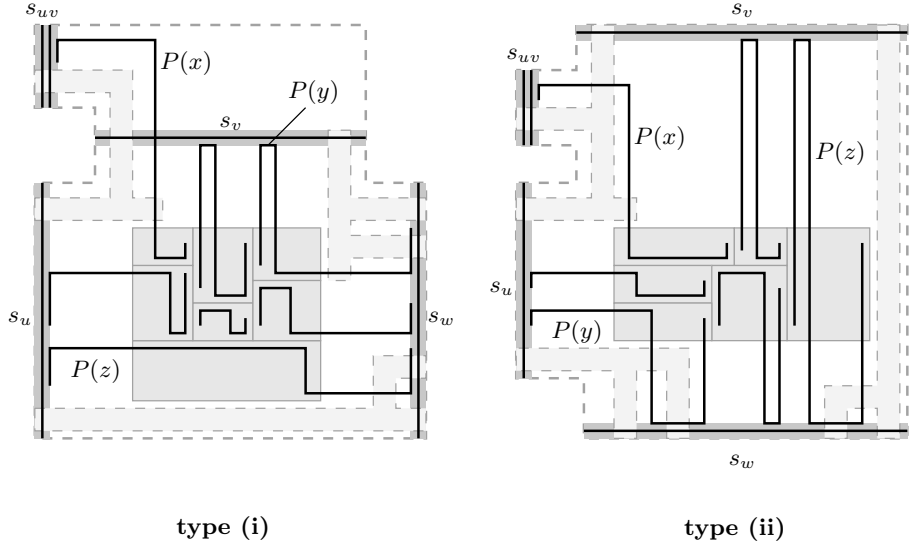}
  \caption{Examples of the $4$-bend representation $\Gamma$ contained in the region for $\Delta$. The new regions of the three facial triangles that contain two vertices from $\Delta$ are highlighted.}
  \label{fig:planar-insert}
 \end{figure}

 Including the already given paths $P(u)$, $P(v)$, and $P(w)$ gives a $4$-bend representation $\Gamma$ of $G$. It is easy to see that $\Gamma$ is indeed representing the graph $G$ and that the above construction can be done in linear time. 

It remains to show that $\Gamma$ satisfies invariant $I$. We have to identify a region for every inner facial triangle of $G$, such that each region lies within the region for $\Delta$ and all regions are pairwise disjoint. Therefore let $\Delta' = \{u',v',w'\}$ be such an inner facial triangle. First, assume that none of $\{u',v',w'\}$ is a vertex of $\Delta$. This is,  $\Delta'$ consists of inner vertices only. Then the intersection $R(u') \cap R(v') \cap R(w')$ of the corresponding rectangles is a single point in the region for $\Delta$. Moreover, exactly one of the three rectangles has its bottom-left or top-right corner at this point. Hence, the paths $P(u')$, $P(v')$ and $P(w')$ locally look like in one of the cases in Fig.~\ref{fig:pla-up-inner}. In the figure, for each case a region for $\Delta' = \{u',v',w'\}$ is highlighted. Note that these regions can be chosen to be pairwise disjoint.

 \begin{figure}[htb]
  \centering
  \scriptsize
  \psfrag{u}[cc][bl]{$P(u')$}
  \psfrag{v}[cc][bl]{$P(v')$}
  \psfrag{w}[cc][bl]{$P(w')$}
  \includegraphics{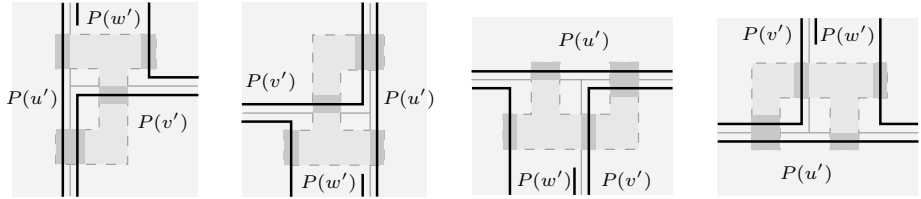}
  \caption{The four possibilities for an inner face $\Delta' = \{u',v',w'\}$ of $G$ with at most one vertex from $\Delta$. The region for $\Delta'$ is highlighted.}
  \label{fig:pla-up-inner}
 \end{figure}

 Similarly, if exactly one of $\{u,v,w\}$ is a vertex in $\Delta'$, then there is a point on the corresponding segment ($s_u$, $s_v$ or $s_w$) where the paths $P(u')$, $P(v')$ and $P(w')$ locally look like in the previous case. Hence in this case we find a region for $\Delta'$, too.

 Finally, exactly three inner facial triangles in $G$ contain exactly two outer vertices. The regions for these three faces are defined as highlighted in Fig.~\ref{fig:planar-insert}. More formally, the region for the triangle $\{u,v,x\}$ contains the segments $s'_u$, $s'_v$, $s'_x$, and $s'_{uv}$ which are contained in $s_u$, $s_v$, the middle vertical segment of $P(x)$, and $s_{uv}$, respectively. The region for the triangle $\{u,w,y\}$ contains the segments $s'_u$, $s'_w$, $s'_y$, and $s'_{wy}$, which are contained in $s_u$, $s_w$, the lower horizontal (type~(i)) or the middle vertical (type (ii)) segment of $P(y)$, and $P(y)\cap s_w$, respectively. Finally, the region for the triangle $\{v,w,z\}$ contains the segments $s'_v$, $s'_w$, $s'_z$, and $s'_{wz}$, which are contained in $s_v$, $s_w$, the lower horizontal (type~(i)) or the rightmost vertical (type (ii)) segment of $P(z)$, and $P(z)\cap s_w$, respectively.

 All these regions are pairwise disjoint and completely contained in the region for $\Delta$. We remark that even in case\footnote{This happens if and only if $G = K_4$.}
$x = y = z$ the above definitions of $4$-bend paths and regions meet the required properties, see~Fig.~\ref{fig:pla-xyz}.

\begin{figure}[htb]
 \psfrag{u}[cc][bl]{$s_u$}
 \psfrag{v}[cc][bl]{$s_v$}
 \psfrag{w}[cc][bl]{$s_w$}
 \psfrag{uv}[cc][bl]{$s_{uv}$}
 \psfrag{x}[cc][bl]{$P(x)$}
 \psfrag{A}[cc][bl]{\textbf{type (i)}}
 \psfrag{B}[cc][bl]{\textbf{type (ii)}}
 \centering
 \includegraphics{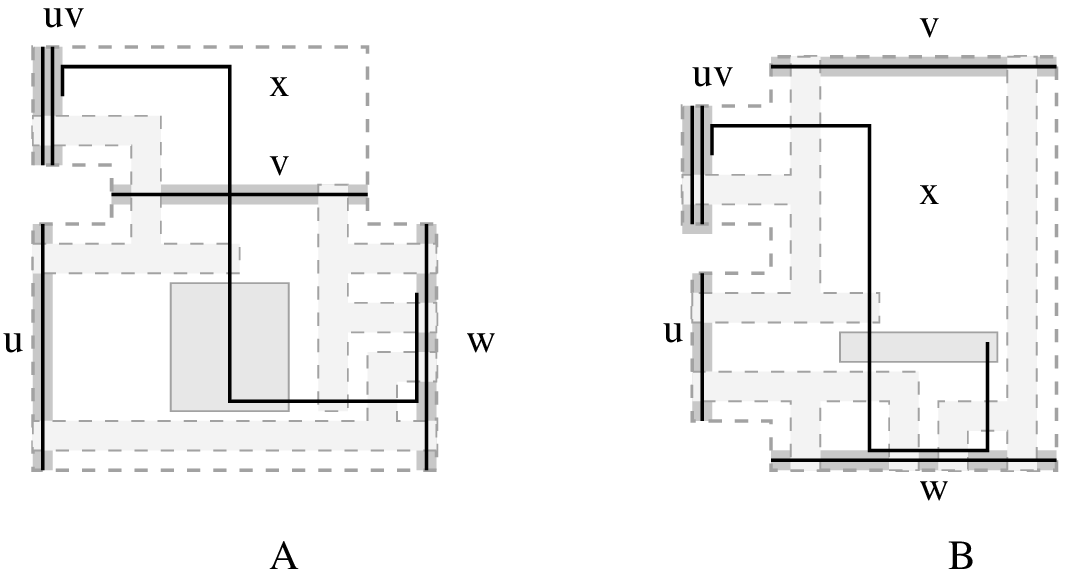}
 \caption{The case $x=y=z$ for a type (i) region and a type (ii) region in the proof of Lemma~\ref{lem:planar-step5}.}
 \label{fig:pla-xyz}
\end{figure}

 Since $G$ does not contain separating triangles, $x$, $y$, and $z$ either all coincide or are pairwise distinct, which completes the proof of the lemma.
\end{proof}

Lemma~\ref{lem:planar-step5} gives a construction as required in step 5.) of our algorithm. Hence, we have proved our main theorem.

\begin{theorem}\label{bendThm:pla-up}
 The bend-number of a planar graph is at most $4$. Moreover, a $4$-bend representation can be found in linear time.
\end{theorem}

Putting Theorem~\ref{bendThm:pla-up} and Proposition~\ref{bendLem:stw3-low} together we have shown the following.

\begin{theorem}\label{thm:planar}
 In an EPG representation of a planar graph $4$-bend paths are always sufficient and $3$-bend paths are sometimes necessary, i.e., the maximum bend-number among all planar graphs is $3$ or $4$.
\end{theorem}


\section{Conclusions}
Although we could raise the previously known lower and upper bound, it remains open to determine the maximum bend-number of planar graphs. 

\begin{conjecture}\label{bendConj:pla-4}
 There is a planar graph $G$, such that every EPG representation of $G$ contains at least one path with four bends.
\end{conjecture}

Another interesting class of representations is as edge-intersection graphs of arbitrary polygonal paths. It is straight-forward to obtain a $4$-bend representation from a representation of touching $\top$s and $\bot$s~\cite{deF-94}. What is the right answer?

\bibliography{lit}
\bibliographystyle{amsplain}

%
%

%
%
%
%
%
%

%
%

%
\end{document}